\documentclass[a4paper]{article}

\usepackage[utf8]{inputenc}
\usepackage{lmodern}
\usepackage{amssymb,amsfonts,amsthm,amsmath,bbm,mathrsfs,aliascnt,mathtools,pifont}
\usepackage[english]{babel}
\usepackage[colorlinks]{hyperref}
\usepackage{tikz}
\usetikzlibrary{decorations}
\usetikzlibrary{decorations.pathreplacing}
\tikzset{individu/.style={draw,thick}}
\newcommand*\circled[1]{\tikz[baseline=(char.base)]{
            \node[shape=circle,draw,inner sep=1pt] (char) {#1};}}

\theoremstyle{plain}
\newtheorem{theorem}{Theorem}[section]
\newtheorem{corollary}[theorem]{Corollary}
\newtheorem{lemma}[theorem]{Lemma}
\newtheorem{proposition}[theorem]{Proposition}

\theoremstyle{definition}

\theoremstyle{remark}
\newtheorem{remark}[theorem]{Remark}

\numberwithin{equation}{section}

\newcommand{\N}{\mathbb{N}}
\newcommand{\Z}{\mathbb{Z}}
\newcommand{\R}{\mathbb{R}}
\newcommand{\C}{\mathbb{C}}

\newcommand{\T}{\mathbf{T}}

\newcommand{\ind}[1]{\mathbf{1}_{\left\{#1\right\}}}

\renewcommand{\bar}[1]{\overline{#1}}

\newcommand{\e}{\mathrm{e}}
\newcommand{\dd}{\mathrm{d}}

\DeclareMathOperator{\E}{\mathbb{E}}

\renewcommand{\P}{\mathbb{P}}

\newcommand{\Y}{\mathbf{Y}}
\newcommand{\PP}{\mathbf{P}}
\newcommand{\EE}{\mathbf{E}}

\renewcommand{\rho}{\varrho}
\renewcommand{\epsilon}{\varepsilon}

\title{Reinforced Galton-Watson processes I:\\  Malthusian exponents}
\author{Jean Bertoin\thanks{Institute of Mathematics, University of Zurich, Switzerland.} \and Bastien Mallein\thanks{LAGA - Institut Galil\'ee, Universit\'e Sorbonne Paris Nord,  France.}}
\date{}
\begin{document}

\maketitle

\begin{abstract} In a reinforced Galton-Watson process with  reproduction law $\boldsymbol{\nu}$ and memory parameter $q\in(0,1)$, the number of children of a typical individual either, with probability
$q$, repeats that of one of its forebears picked uniformly at random, or, with complementary probability  $1-q$, is given by an independent sample from $\boldsymbol{\nu}$. We estimate the average size of the population at a large generation, and in particular, we determine explicitly the  Malthusian growth rate
in terms of $\boldsymbol{\nu}$ and $q$. Our approach via the analysis of transport equations  owns much to works by Flajolet and co-authors.

\end{abstract}


\medskip


\section{Introduction and main result}
\label{sec:introduction}
We present first some motivations for this work, then our main result, and finally, we discuss techniques and sketch the plan of the rest of the article.
\subsection{Motivations}
Roughly speaking, stochastic  reinforcement refers to step-by-step modifications of the dynamics of a random process such that transitions that already occurred often  in the past are more likely to be repeated in the future.
The evolution  thus depends on the past of the process and not just on its current state, and is typically non-Markovian\footnote{Nonetheless, it is a remarkable fact observed first by Coppersmith and Diaconis that  linearly edge-reinforced random walks on a finite graph are actually mixtures of reversible Markov chains, with the mixing measure  given by the so-called ``magic formula''.}. We refer to  \cite{Pem} for a survey of various probabilistic models in this area, and merely recall that the concept of reinforcement notably  lies at the heart of machine learning. The general question of describing
how reinforcement impacts the long time behavior of processes has been intensively investigated for many years. In particular, there is a rich literature
on so-called linear  edge/vertex reinforced random walks, which culminates with the remarkable achievement  \cite{SabZen} in which many important earlier references can also be found.
In short, the purpose of the present work is to investigate the effects of reinforcement on branching processes.

More precisely, fix a reproduction law $\boldsymbol{\nu}$, that is $\boldsymbol{\nu}=(\nu(k))$ is a probability measure on $\Z_+=\{0,1, \ldots\}$. In a classical (Bienaym\'e-)Galton-Watson process, every individual reproduces independently of the others, and $\nu(k)$ is the probability that a typical individual has $k$ children at the next generation. The reinforced version that we took interest in
involves random repetitions of reproduction events, much in the same way as words are repeated in a well-known algorithm of Herbert Simon \cite{Simon}, or  steps  in a family of discrete time processes with memory  known as step-reinforced random walks (see \cite{BerRos, Bscaling, Buniver, Laulinthese} and references therein). Specifically, the reinforced evolution
depends on a memory parameter $q\in(0,1)$. For every individual, say $i$,  at a given generation $n\geq 1$, we  pick a forebear, say $f(i)$, uniformly at random on its ancestral lineage.
That is, $f(i)$ is the ancestor of $i$ at generation $u(n)$, where $u(n)$ is a uniform random sample from $\{0,\ldots, n-1\}$. Then, either with probability $q$, $i$ begets the same number of children as  $f(i)$, or with complementary probability $1-q$,  the number of children of $i$ is given by an independent sample from the reproduction law $\boldsymbol{\nu}$.
Implicitly,  for different individuals $i$, we use  independent uniform variables for selecting forebears $f(i)$ on their respective ancestral lineages, and the repetition events where given individuals $i$ reproduce as their respective selected forebears $f(i)$ are also assumed to be independent.

The reproduction of  an individual in a reinforced Galton-Watson process depends on that along  its entire ancestral lineage.
 It should be clear that, with the exception of the case of a few elementary reproduction laws,  this  invalidates fundamental features of Galton-Watson processes.  Notably, different individuals partly share the same ancestral lineage, and thus their descent are not independent. The Markov and the branching properties are lost, and even the simplest questions about Galton-Watson processes become challenging in the reinforced setting.

One naturally expects that reinforcement should enhance of the growth of branching processes. Let us briefly  analyze the elementary case of a binary branching when $\nu(2)=p$ and $\nu(0)=1-p$ for some $p\in(0,1)$. Plainly, for every non-empty generation, all the forebears  on any ancestral lineage had $2$ children, and therefore  the reinforced Galton-Watson process is in turn a binary Galton-Watson process, such that every individual aside the ancestor has probability $q+(1-q)p>p$ of having two children. In particular, the reinforced process survives forever with positive probability if and only if $q+(1-q)p>1/2$, which includes some subcritical reproduction laws with $p<1/2$ (for example it can survive for all $p > 0$ whevener $q>1/2$).

In the next section, we describe our main result on how reinforcement affects the Malthusian exponent, that is the growth rate of the averaged size of the population.
In another forthcoming work, we shall investigate in turn the impact of reinforcement on survival probabilities.

\subsection{Main result}
We first introduce some notation; the memory parameter $q\in(0,1)$ and the reproduction law $\boldsymbol{\nu}$ have been fixed and thus most often are omitted from the notation. The degenerate case when $\nu(k)=1$ for some $k\in \Z_+$ will be implicitly excluded.
We write  $Z=(Z(n))_{n\geq 0}$  for  the reinforced Galton-Watson process, where $Z(n)$ is the size of population at the $n$-th generation. We agree for simplicity that there is a single ancestor at generation $0$ (i.e. $Z(0)=1$) which has $Z(1)$ children, and denote  the conditional distribution of the reinforced Galton-Watson process given $Z(1)=\ell$ by $\P_\ell$.
We are mainly interested in the situation where $Z(1)$ is random and distributed according to the reproduction law ${\boldsymbol{\nu}}$, and we shall then simply write
  $\P=\P_{\boldsymbol{\nu}}$.

Perhaps the most basic question about $Z$  is to estimate the average size of the population for large times; let us start with a few elementary observations.
Reinforcement entails that the conditional probability for an individual to
 have $k$ children, given that each of  its ancestors also had $k$ children,
 equals $q+(1-q)\nu(k)$. By focussing on such individuals, we see that there is a sub-population excerpted from the reinforced Galton-Watson process that evolves as an ordinary Galton-Watson process with averaged reproduction size  equal to  $k(q+(1-q)\nu(k))> kq$. This points at the roles of the support of the reproduction law $\boldsymbol{\nu}$,
$$\mathrm{Supp}(\boldsymbol{\nu})\coloneqq \{k\geq 0: \nu(k)>0\},$$
  and of the  maximal possible number of children which we denote by
$$
\mathrm{k}^*\coloneqq \sup\{k\geq 0: \nu(k)>0\}.
$$
Namely, if $q \mathrm{k}^*\geq 1$, then not only does the reinforced Galton-Watson process survive with positive probability, but the  averaged population size grows at least exponentially fast with $\E(Z(n)) \geq c(q\mathrm{k}^*)^n$ for some $c>0$. We stress that a high averaged growth rate is achieved as soon as the reproduction law has an atom at some  large $k$, no matter how small the mass $\nu(k)$ of this atom is.

We  assume henceforth that the support of the reproduction law is bounded, viz.
$$\mathrm{k}^*<\infty,$$
and stress that otherwise, the averaged size of the population would grow super-exponentially fast.
 To state our main result, we  introduce the function
\begin{equation} \label{E:pi}
\Pi(t)\coloneqq  \prod (1-tk)^{\nu(k)(1-q) /q},  \qquad t\leq 1/\mathrm{k}^*,
\end{equation}
where the product in the right-hand side is taken over all integers $k$'s (actually, only the positive $k$'s in $\mathrm{Supp}(\boldsymbol{\nu})$ contribute to this product).
Observe that $\Pi(0)=1$, $1/\mathrm{k}^*$ is the sole zero of $\Pi$, and $\Pi$ is strictly decreasing on $[0,\frac{1}{\mathrm{k}^*}]$. We can then define
\begin{equation}\label{E:I}
m_{\boldsymbol{\nu},q}\coloneqq \frac{q}{\int_0^{1/\mathrm{k}^*} \Pi(t) \dd t}.
\end{equation}
In particular, as $\Pi(t) < 1$ for all $0< t \leq 1/\mathrm{k}^*$, we remark that $m_{\boldsymbol{\nu},q} >q\mathrm{k}^*$.

\begin{theorem}\label{T1} The Malthusian exponent of the reinforced Galton-Watson process $Z$ equals $\log m_{\boldsymbol{\nu},q}$, in the sense that
$$\lim_{n\to \infty} \frac{1}{n} \log \E (Z(n)) = \log m_{\boldsymbol{\nu},q}.$$
Much more precisely,  we have
$$ m_{\boldsymbol{\nu},q}^{-n}\E (Z(n)) =   \frac{ \nu(\mathrm{k}^*)}{q+
 \nu(\mathrm{k}^*)(1-q)} +O(n^{-q/(q+(1-q)\nu(\mathrm{k}^*))}).$$
\end{theorem}

Plainly, if $Z(n)=0$ for some $n\geq 1$, then also $Z(n')=0$ for all $n'\geq n$ and one says that the population becomes extinct eventually. Applying the Markov inequality to $Z(n)$, we remark that $m_{\boldsymbol{\nu},q} < 1$ implies that  almost surely, the reinforced Galton-Watson process becomes extinct eventually. Therefore the condition $m_{\boldsymbol{\nu},q} \geq 1$ is necessary for the survival of the reinforced Galton-Watson process with strictly positive probability. We conjecture however that this condition should not be sufficient; we believe that there should exist reproduction laws $\boldsymbol{\nu}$ and memory parameters $q$ such that, almost surely, the reinforced Galton-Watson process becomes extinct eventually, even though $\E(Z(n)) \to \infty$ as $n \to \infty$.

\subsection{Techniques and plan}
Although reinforcement invalidates fundamental properties of Galton-Watson processes, our strategy for establishing Theorem~\ref{T1} nonetheless uses some classical tools from the theory of branching processes, and notably analytic methods. These techniques can be pushed further and yield an asymptotic expansion of $\E (Z(n))$ with an arbitrary order. It is  for the sake of readability only that we just consider the first order in this work.

In Section 2, we start by relating  the expected population size for the reinforced Galton-Watson process to the factorial moment generating functions of a multitype Yule process; see Lemma~\ref{L1}. In short, this combines the basic many-to-one formula
with the classical embedding of urn schemes into continuous time multitype branching processes (see, for instance, \cite[Section V.9]{AN}).

In Section 3, we observe that those factorial moment generating functions verify a multidimensional ODE; see Lemma~\ref{L3} there. This is similar to the use of analytic methods  that were introduced in the pioneering works of Flajolet \textit{et al.} \cite{Flajoes,Flajoau} for determining distributions in certain P\'olya type urn processes.

In Section 4, we shall perform several transformations to reduce the multidimensional ODE obtained above to a one-dimensional inhomogeneous transport equation with growth. We then
 analyze  characteristics curves and in particular, we determine the parameters for which the factorial moment generating functions explode; see Proposition~\ref{P1}. This points at a critical case, which corresponds to the situation when the quantity $m_{\boldsymbol{\nu},q}$ in \eqref{E:I} equals $1$.

Section 5 is then devoted to the study of the asymptotic behavior of the factorial moment generating functions in critical case. Our main result there, Proposition~\ref{P2},  shows that this problem fits the framework of  singular analysis of generating functions; see \cite{FlajOdl,FlajSed}.

Theorem~\ref{T1} is established in Section 6; it suffices to combine the elements developed in the preceding sections. Last, some comments and further results are presented in Section 7.
Notably, we discuss a sharper version of Theorem~\ref{T1} under the law $\P_\ell$, that is given that the ancestor has $\ell$ children.

\section{A multitype Yule process}
We introduce a standard Yule process $Y=(Y(t))_{t\geq 0}$, that is $Y$ is a pure birth process started from  $Y(0)=1$ and such that the rate of jump from $k$ to $k+1$ equals $k$ for all $k\geq 1$. We  think of $Y$ as the process recording the size of a population evolving in continuous time, such that individuals are immortal and each individual gives birth to a child at unit rate and independently of the other individuals. The interpretation in terms of population models, is that at any birth event, the parent  that begets is chosen uniformly at random
in the population immediately before the birth event.
 We enumerate individuals according to  the increasing order of their birth times and then record the genealogy as a combinatorial  tree $\T$ on $\N=\{1, \ldots \}$. The  distribution of  $\T$  is that of an infinite random recursive tree, that is such that for every fixed $n\in \N$, its restriction to
 $\{1, \ldots, n\}$ has a the uniform distribution on the set of trees rooted at $1$ and such that the sequence of vertices along any branch from the root is increasing. Clearly enough, the  tree $\T$ and the Yule process $Y$ are independent (the latter only records the size of the population as time passes and ignores the genealogy).

Next, we randomly and recursively  assign types to the individuals in the population, that is to the vertices of $\T$, where the space of types is $\mathrm{Supp}(\boldsymbol{\nu})$, the support of the reproduction law $\boldsymbol{\nu}$.  For every $n\geq 2$, conditionally on the genealogical tree $\T$ and the types already assigned to the $n-1$ first individuals, if the type of  the parent of $n$ is $j$, then the probability that the type of $n$ is $k$ equals
$$q\ind{k=j} + (1-q)\nu(k).$$
We should think of the assignation of types as a random process on $\T$, such that for each individual $i\geq 2$ (the ancestor $1$ is excluded), with probability $q$ the type of $i$ merely repeats that of its parent, and with complementary probability $1-q$, it is chosen at random according to the reproduction law $\boldsymbol{\nu}$.
We stress that the assignation of types is performed independently of the Yule process $Y$.
If originally, the type $\ell$ is assigned to the ancestor $1$, then we write $\PP_{\ell}$ for the distribution of the population process with types, and $\PP_{\boldsymbol{\nu}}$, or simply $\PP$, when the type of $1$ is random with law $\boldsymbol{\nu}$.

For every $t\geq 0$, let us write $Y_j(t)$ for the number of individuals of type $j$ in the population at time $t$, and $\Y(t)=(Y_j(t))$ for the sequence indexed by $j\in \mathrm{Supp}(\boldsymbol{\nu})$.
Then the process $\Y=(\Y(t))_{t\geq 0}$ is  a multitype Yule process such that every individual begets a child bearing the same type with rate $q$, and a child with type chosen independently according to $\boldsymbol{\nu}$ with rate $1-q$.

We point at the similarity, but also the difference, with the reinforcement dynamic that motivates this work.
In both cases, the new born individual repeats with probability $q$ the type, or the offspring number, of a forebear, and is otherwise assigned an independent type, or offspring number, distributed according to the  fixed law $\boldsymbol{\nu}$. However, the forebear is merely the parent of the new individual in the multitype Yule setting, whereas it is  picked uniformly
at random on the ancestral lineage in the reinforced Galton-Watson setting.

We can now state the connexion between the reinforced Galton-Watson process and the multitype Yule process which is the first step of our analysis.
In this direction, recall that  $\P_\ell$ stands for the conditional distribution of the reinforced Galton-Watson process given $Z(1)=\ell$.
\begin{lemma}\label{L1} For every $t\geq 0$, $c>0$, and $\ell\in \mathrm{Supp}(\boldsymbol{\nu})$, there is the identity
$$\EE_{\ell}\left( \prod (cj)^{Y_j(t)}\right) = \e^{-t}\sum_{n=1}^{\infty} (1-\e^{-t})^{n-1} c^n\E_{\ell}(Z(n)),$$
where in the left-hand side, the product implicitly runs over $j\in \mathrm{Supp}(\boldsymbol{\nu})$ and we use the convention  $0^0=1$.
\end{lemma}

\begin{proof}
Consider any deterministic genealogical tree $\mathcal{T}$ and construct iteratively a sequence of random variables $(\zeta_n)$ recording the sibling sizes along a randomly chosen descent lineage. Specifically, we first set $\zeta_1$ as the number of children of the root of $\mathcal{T}$. If $\zeta_1 = 0$, then we also set $\zeta_n = 0$ for all $n \geq 2$. Otherwise, we select an individual chosen uniformly at random  in the first generation, and write $\zeta_2$ for the number of its children at the second generation, and iterate in an obvious way. That is, if  $\zeta_2=0$, then  $\zeta_n=0$ for all $n\geq 3$, and else, we select a child uniformly at random in the sibling of size $\zeta_2$, ... We then observe that for any $n \in \N$, writing $\mathcal{Z}(n)$ for the number of individuals at the $n$th generation of $\mathcal{T}$, we have
\[
  \mathcal{Z}(n) = \E(\zeta_1\times \cdots \times \zeta_n).
\]
Indeed, observe that the probability for a given individual  in the $n$th generation  to be chosen is equal to the inverse of the product of the sibling sizes along its ancestral path.

In particular, in the case of the reinforced Galton-Watson process, we get
\begin{equation}\label{E:many21}\E_\ell(Z(n)) = \E_\ell(\zeta_1\times \cdots \times \zeta_n).
\end{equation}
The reinforcement dynamics entail that, provided that $\zeta_n\geq 1$, the conditional probability that $\zeta_{n+1}=k$ given the first $n$ values $(\zeta_i)_{1\leq i \leq n}$ is given by
$$q n^{-1} \sum_{i=1}^n \ind{\zeta_i=k} + (1-q)\nu(k).$$
Comparing with the evolution of the multitype Yule process $\Y$, we realize that for any $n\geq 1$, the conditional distribution of $(\zeta_i)_{1\leq i \leq n+1}$ given $\zeta_1=\ell$ and $\zeta_n\neq 0$ is the same as the conditional distribution under $\PP_\ell$ of the sequence of the types of the first $n+1$ individuals of the multitype Yule process  given that no individuals amongst the first $n$ have type $0$.
Moreover, we have also
$$\P_\ell(\zeta_n=0)=\PP_{\ell}(Y_0(\beta_n)\geq 1),$$
where $\beta_n=\inf\{t\geq 0: Y(t)=n\}$ stands for the birth time of the $n$-th individual in the multitype Yule process.
As a consequence, we have
$$\E_\ell(\zeta_1\times \cdots \times \zeta_n) = \EE_\ell\left( \prod j^{Y_j(\beta_n)}\right),$$
where, as usual, the product in the right-hand side runs over $j\in \mathrm{Supp}(\boldsymbol{\nu})$. On the other hand, recalling that the monotype Yule process $Y= \sum Y_j$ is independent of the assignation of types, we get that for any $c$ and $t>0$, there is the identity
$$c^n\EE_\ell\left( \prod j^{Y_j(\beta_n)}\right) = \EE_\ell\left( \prod(c j)^{Y_j(t)}\mid Y(t)=n\right),$$

Putting the pieces together, we arrive at
$$\EE_\ell\left( \prod (cj)^{Y_j(t)}\mid Y(t)=n\right)= c^n\E_\ell(Z(n)).$$
We can now conclude the proof by recalling the basic fact that $Y(t)$ has the geometric distribution with parameter $\e^{-t}$.
\end{proof}

We conclude this section by pointing at monotonicity property in the memory parameter $q$, which may be rather  intuitive even though its proof is not completely obvious. We also stress that this property is valid
under $\P=\P_{\boldsymbol{\nu}}$, but might fail under some $\P_{\ell}$.
\begin{lemma} \label{L-1} For every fixed $n\geq 1$, the average population size $\E(Z(n))$ is a monotone increasing function of the memory parameter $q$.
\end{lemma}
\begin{proof} We will use a coupling argument for a simple variation of P\'olya  urn model with random replacement and balls labelled  $\circled{\small{1}}, \circled{\small{2}}, \ldots$. Imagine an urn which contains initially just one ball $\circled{\small{1}}$, and let $(\epsilon_n)_{n\geq 1}$ a sequence of i.i.d. Bernoulli variables with parameter $q$. At each step $n=1, \ldots$, we draw a ball uniformly at random from the urn and return it together with, either a copy of the sampled ball   if $\epsilon_n=1$, or  a new ball bearing the first label which was not already present in the urn if $\epsilon_n=0$.
For any $k\geq 1$, denote by $N_k(n)$ the number of balls $\circled{\small{$k$}}$ in the urn after $n-1$ steps, i.e. when the urn contains in total $n$ balls.

Then consider also  a sequence $(\xi_k)_{k\geq 1}$ of i.i.d. samples from $\boldsymbol{\nu}$, independent of the urn process. We get from the many-to-one formula \eqref{E:many21} (recall the convention $0^0=1$) that
$$\E(Z(n)) = \E\left( \prod_{k\geq 1} \xi_k^{N_k(n)}\right) = \E\left( \prod_{k\geq 1} \mathrm{m}_{\boldsymbol{\nu}}(N_k(n))\right),$$
where in the right-hand side, we wrote
$$\mathrm{m}_{\boldsymbol{\nu}}(\ell) \coloneqq \E(\xi^\ell)= \sum_j j^\ell \nu(j), \qquad \ell \geq 0$$
 for the $\ell$-th moment of the reproduction law.

 Now let $q'>q$. By an elementary coupling, we get  a sequence  $(\epsilon'_n)_{n\geq 1}$ of i.i.d. Bernoulli variables with parameter $q'$ such that  $\epsilon'_n\geq \epsilon_n$ for all $n\geq 1$.
We can use these to couple the urn processes with respective memory parameters $q$ and $q'$, so that the second one is obtained from the first by some (random) non-injective relabelling of the balls. Since we know from Jensen's inequality that
 $$ \mathrm{m}_{\boldsymbol{\nu}}(\ell) \times \mathrm{m}_{\boldsymbol{\nu}}(\ell')\leq \mathrm{m}_{\boldsymbol{\nu}}(\ell+\ell'),$$
 we deduce that
 $$ \prod_{k\geq 1} \mathrm{m}_{\boldsymbol{\nu}}(N_k(n)) \leq  \prod_{k\geq 1} \mathrm{m}_{\boldsymbol{\nu}}(N'_k(n)),$$
 where in the right-hand side, $N'_k(n)$ stands for the number of balls $\circled{\small{$k$}}$ in the urn with memory parameter $q'$ after $n-1$ steps.
 Taking expectations completes the proof.
\end{proof}

\section{Factorial moment generating functions of the multitype Yule process}
Lemma~\ref{L1} incites us to investigate the factorial moment generating functions ${\boldsymbol{M}}=(M_{\ell})$ of the multitype Yule process $\Y$, where
$$M_{\ell}(\boldsymbol{a},t)\coloneqq \EE_\ell \left( \prod {a}_j^{Y_j(t)}\right) , \qquad t\geq 0 \text{ and } \boldsymbol{a}=({a}_j)\in \R_+^{\mathrm{Supp}(\boldsymbol{\nu})}.
$$
Clearly, $M_{\ell}(\boldsymbol{a},t)\equiv 0$ if ${a}_{\ell}=0$, and otherwise is a strictly positive quantity which may be infinite when $\| \boldsymbol{a}\|_{\infty}>1$
and $t$ is sufficiently large, with the notation
$\|\cdot \|_{\infty}$ for the maximum norm on $\R^{\mathrm{Supp}(\boldsymbol{\nu})}$. We first make the following simple observation in this direction. From now on, the trivial case when  $\boldsymbol{a}=0$ will be systematically ruled out.
\begin{lemma} \label{L2} For every  $\boldsymbol{a}\in \R_+^{\mathrm{Supp}(\boldsymbol{\nu})}$,  there exists  $\rho(\boldsymbol{a})\in (0,\infty]$
such that for every $\ell\in \mathrm{Supp}(\boldsymbol{\nu})$ with ${a}_{\ell}>0$,
$$M_{\ell}(\boldsymbol{a},t)
\left\{
\begin{matrix}
<\infty &
 \text{ for all }0\leq t <\rho(\boldsymbol{a}) ,\\
=\infty & \text{ for all } t>\rho(\boldsymbol{a}).
\end{matrix} \right. $$
Moreover,
$t\mapsto M_{\ell}(\boldsymbol{a},-\log(1-t))$ is  an entire function with radius of convergence no less than $1-\e^{-\rho(\boldsymbol{a})}$.

\end{lemma}
\begin{remark}
For instance, if ${a}_j=c>1$ for all $j$,
then the factorial moment generating functions of the multitype Yule process  $\Y$ reduce to that of the monotype Yule process $Y$, and we immediately get
$$M_{\ell}(\boldsymbol{a},t) =\frac{ c\e^{-t}}{1-c(1-\e^{-t})}\qquad \text{for }\quad t< \rho_{\ell}(\boldsymbol{a})=\log(1-1/c).$$
\end{remark}

\begin{proof} Using the independence of the Yule process and the assignation of types as  at the end of the proof of Lemma~\ref{L1}, we
see that
$$M_{\ell}(\boldsymbol{a},t)= \e^{-t}f(1-\e^{-t}),$$
where $f(x)=\sum_{n=0}^{\infty} a_n x^n$ is some entire series with positive coefficients (for notational simplicity, we do not indicate the dependency in $\ell$).
Writing $r>0$
for the radius of convergence of $f$ and then setting
$$\rho_{\ell}(\boldsymbol{a})\coloneqq \left\{
\begin{matrix} -\log(1-r) &\text{ if }r < 1,\\
\infty &\text{ if }r \geq 1,
\end{matrix}
\right.
$$
yields our claims provided that $\rho_{\ell}(\boldsymbol{a})$ does not depend on the index $\ell$.

 It thus only remains to check that $\rho_{\ell}(\boldsymbol{a})\geq \rho_{j}(\boldsymbol{a})$ for any $j\in \mathrm{Supp}(\boldsymbol{\nu})$ with ${a}_{j}>0$. We may assume that $\rho_{\ell}(\boldsymbol{a})<\infty$ since otherwise there is nothing to prove. Take any $t'>t>\rho_{\ell}(\boldsymbol{a})$, and  work under the law $\PP_j$ when the type of ancestor is $j$. The probability $p(t'-t)$ that
at time $t'-t$ the population has exactly two individuals, one with type $j$ and the other with type $\ell$, is strictly positive. The branching property then shows that
$$M_{j}(\boldsymbol{a},t')\geq p(t'-t) M_{\ell}(\boldsymbol{a},t) M_{j}(\boldsymbol{a},t)=\infty.$$
Hence $ \rho_{j}(\boldsymbol{a})\leq t'$, and the proof is now complete.
\end{proof}

We shall henceforth refer to
$\rho(\boldsymbol{a})$ as the explosion time. The problem of determining its value  is at the heart of the proof of Theorem~\ref{T1}. For instance, recall from the Cauchy–Hadamard theorem and Lemma~\ref{L1} that, taking ${a}_j=cj$ for all $j \in \mathrm{Supp}(\boldsymbol{\nu})$ and any $c$ sufficiently large so that $\rho(\boldsymbol{a})<\infty$, we obtain the identity
$$ \limsup_{n\to \infty} n^{-1}\log \E(Z(n))= -\log(1-\e^{-\rho(\boldsymbol{a})}) -\log c.$$
The first step for its resolution is closely related to the analytic approach\footnote{Note however that these works deal with moment generating functions, whereas  we rather consider here their factorial version.} of two-color P\'olya urn models by Flajolet \textit{et al.} \cite{Flajoes,Flajoau}, see also \cite{Inoue} for a version with random replacement schemes.
It  relies on the observation
that the factorial moment generating functions solves a multidimensional ODE.

We use the notation $\langle \cdot; \cdot \rangle$ for the scalar product on $\R^{\mathrm{Supp}(\boldsymbol{\nu})}$ and set
\begin{equation} \label{E:varphi}
\varphi(t)\coloneqq  (1-q) \left\langle \boldsymbol{\nu}; {\boldsymbol{M}}(\boldsymbol{a},t)\right \rangle-1.
\end{equation}
 For every $\ell\in \mathrm{Supp}(\boldsymbol{\nu})$ and $t< \rho(\boldsymbol{a})$, we also write
$$M'_{\ell}(\boldsymbol{a},t) \coloneqq \frac{\partial M_{\ell}}{\partial t} (\boldsymbol{a},t).
 $$

\begin{lemma} \label{L3} For every  $\boldsymbol{a}$, the  function ${\boldsymbol{M}}(\boldsymbol{a}, \cdot)$
 is the unique solution  on $[0,\rho(\boldsymbol{a}))$ to,
$$M'_{\ell}(\boldsymbol{a},t)= M_{\ell}(\boldsymbol{a},t)\left( q M_{\ell}(\boldsymbol{a},t)+\varphi(t)\right),\qquad \ell\in \mathrm{Supp}(\boldsymbol{\nu}),$$
with initial condition ${\boldsymbol{M}}(\boldsymbol{a}, 0)=\boldsymbol{a}$.
\end{lemma}

Remark that the differential equation satisfied by $\log M$ is analogue to the one solved by the Laplace exponent associated to continuous state branching processes; see e.g. \cite[Theorem 12.1]{Kypr}.

\begin{proof}
We denote by $T$ the first branching time of the multitype Yule process and by $L$ the random label of the newly added vertex. Applying the branching property at time $T$, we have
\begin{align*}
 M_\ell(\boldsymbol{a},t) &= a_\ell \P(T > t) + \E\left( M_\ell(\boldsymbol{a},t-T) M_L(\boldsymbol{a},t-T) \right)\\
 &= a_\ell e^{-t} + \int_0^t e^{-s} M_\ell(\boldsymbol{a},t-s) \left(q M_\ell(\boldsymbol{a},t-s) + \varphi(t-s) + 1\right) \dd s.
\end{align*}
Hence, by change of variable $u = t-s$, we obtain
\[
  e^t M_\ell(\boldsymbol{a},t) = a_\ell + \int_0^t e^u M_\ell(\boldsymbol{a},u) \left(q M_\ell(\boldsymbol{a},u) + \varphi(u) + 1\right).
\]
Taking the derivative of this expression in $t$ completes the proof.
\end{proof}

We finish this section by pointing at simple monotonicity properties of the factorial moment generating functions and their explosion times.

\begin{corollary} \label{C1} (i) For every $j,\ell\in \mathrm{Supp}(\boldsymbol{\nu})$ with
 $ {a}_j\leq {a}_\ell $,  the ratio function
 $$t\mapsto \frac{M_\ell(\boldsymbol{a},t)}{ M_j(\boldsymbol{a},t)}$$
 is monotone non-decreasing  on $[0, \rho(\boldsymbol{a}))$. In particular, if $ {a}_j ={a}_\ell $, then for all $t \geq 0$ we have $M_j(\boldsymbol{a},t)= M_\ell(\boldsymbol{a},t) $.

 (ii) Suppose that  $$ \sup_{t\geq 0}M_\ell(\boldsymbol{a},t) < \infty.$$
 Then for all $j\in \mathrm{Supp}(\boldsymbol{\nu})$ with $ {a}_j < {a}_\ell $, we have
 $$\lim_{t\to \infty} M_j(\boldsymbol{a},t) =0. $$
 \end{corollary}
\begin{proof} The claims  readily derive from the observation from Lemma~\ref{L3} that the logarithmic derivative of the ratio function equals $q(M_\ell(\boldsymbol{a},t)- M_j(\boldsymbol{a},t))$.
\end{proof}

\section{Determining the factorial moment generating function}
The purpose of this section is to show that  the factorial moment generating functions ${\boldsymbol{M}}(\boldsymbol{a},t)$ can be determined rather explicitly, and this will be sufficient for our purposes. To do so, it is convenient to introduce a bivariate moment-generating function  $G$ defined by
\begin{equation}
  \label{eqn:defG}
  G(t,s) \coloneqq    \sum_j  \frac{M_j(\boldsymbol{a},t)}{1 - s M_j(\boldsymbol{a},t)} \nu(j)
\end{equation}
for any $ 0\leq t  <\rho(\boldsymbol{a})$ and
$ 0\leq s<1/\|{\boldsymbol{M}}(\boldsymbol{a},t)\|_{\infty}$.
Note that
$$ G(t,0)= \left\langle \boldsymbol{\nu}; {\boldsymbol{M}}(\boldsymbol{a},t)\right \rangle = \EE_{ \boldsymbol{\nu}} \left( \prod {a}_j^{Y_j(t)}\right)$$
is the function with explosion time $\rho(\boldsymbol{a})$ that we would like to determine, and that
\begin{equation} \label{E:G(0,s)}
 G(0,s) = \sum_j   \frac{{a}_j}{1 - s {a}_j} \nu (j)
 \end{equation}
is a known function of the variable $s$.

The motivation for introducing $G$ is that it satisfies  an inhomogeneous transport equation with growth in dimension $1$.
\begin{lemma} \label{L4}
\label{lem:pdeForG} For all $ 0\leq t  <\rho(\boldsymbol{a})$ and $0\leq s<1/\|{\boldsymbol{M}}(\boldsymbol{a},t)\|_{\infty}$,
the function $G$ satisfies
\[
  \partial_t G (t,s)= (q + s\varphi(t)) \partial_s G(t,s) + \varphi(t) G(t,s),
\]
where $\varphi(t)$ has been defined in \eqref{E:varphi}.
\end{lemma}

\begin{proof}
For all $k \geq 1$, we introduce the function $ R_k : [0,\rho(\boldsymbol{a}))\to \R_+$ by
\[
 R_k(t)\coloneqq
   \sum_j \nu(j) M_j(\boldsymbol{a},t)^k.
\]
The function $G(t,\cdot)$ arises in this setting as the moment-generating function of the sequence $(R_k(t))_{k\geq 1}$. Namely, we have  for every $ 0\leq s<1/\|{\boldsymbol{M}}(\boldsymbol{a},t)\|_{\infty}$ that
\[
  \sum_{k=1}^\infty s^{k-1} R_k(t)  =   \sum_j \nu(j)  \sum_{k=1}^\infty s^{k-1} M_j(\boldsymbol{a},t)^k = G(t,s) . \]
Note also that for $k=1$, we have
$$R_1(t)= \left\langle \boldsymbol{\nu}; {\boldsymbol{M}}(\boldsymbol{a},t)\right \rangle  = G(t,0).$$

We compute the derivative of $R_k$ using Lemma~\ref{L3} at the second identity below, and get
\begin{align*}
  R_k'(t) &= k\sum_j \nu(j) M_j(\boldsymbol{a},t)^k \frac{M'_j(\boldsymbol{a},t)}{M_j(\boldsymbol{a},t)} \\
  &= k \left( qR_{k+1}(t) +\varphi(t) R_k(t) \right).
\end{align*}
This yields
\begin{align*}
  \partial_t G(t,s) &=  \sum_{k=1}^\infty k  s^{k-1}  \left( qR_{k+1}(t) +\varphi(t) R_k(t) \right) \\
   &= \partial_s \left(\sum_{k=1}^\infty s^k \left( qR_{k+1}(t) + \varphi(t)R_k(t) \right)\right) \\
  &= \partial_s \left( q G(t,s) + \varphi(t)s G(t,s) \right).
\end{align*}
We complete the proof from the chain rule.
\end{proof}

Lemma~\ref{L4} may not look quite satisfactory, because the coefficients of the PDE there involve the unknown function $\varphi$ defined in \eqref{E:varphi}, which precisely the one which we would like to determine. However, its main input   is that this PDE can be analyzed in terms of a trajectory on the real line, as opposed to the multidimensional ODE of Lemma~\ref{L3}.
Lemma~\ref{L4} shows indeed that the function $G$ defined in \eqref{eqn:defG} satisfies an inhomogeneous  transport equation, and as a consequence,
$$\varphi(t)=(1-q)G(t,0)-1$$ can be obtained as the solution of an integral equation as we shall now see.
\begin{lemma} \label{L:A}
\label{lem:transport} Introduce the functions
\[ A: [0,\rho(\boldsymbol{a}))\to \R_+, \qquad
  A(t) = \int_0^t \exp\left( \int_0^s\varphi(u)  \dd u \right) \dd s,
\]
and
\begin{equation}\label{E:F}
 F: [0, 1/(q\| \boldsymbol{a}\|_{\infty}))\to \R_+, \qquad
F(s) = (1-q) \sum_{j} \frac{ {a}_j }{1- qs {a}_j}\nu(j).\end{equation}
Then for  $t_0>0$ sufficiently small, $A$ is the unique solution  on $[0,t_0]$ of the Cauchy problem
\begin{equation}
  \label{eqn:diffA}
  A''(t) =  A' (t) \left( A' (t) F(A(t)) - 1\right)\quad \text{ with } A(0) = 0 \text{ and } A'(0) = 1.
\end{equation}
\end{lemma}
\begin{proof}
The method of characteristics applied to the PDE of Lemma~\ref{L4}  incites us  to construct a pair of functions
$$\gamma :  [0,t_0]\to \R_+ \quad \text{and} \quad \sigma :  [0,t_0]\to \R_+,$$
 such that $\sigma(t)< 1/\|{\boldsymbol{M}}(\boldsymbol{a},t)\|_{\infty}$ and
 \begin{equation}
  \label{eqn:fs}
 \gamma(t) G(t,\sigma(t)) = \gamma(0) G(0,\sigma(0)) \qquad \text{for all }0\leq t \leq  t_0,
\end{equation}
where  $0<t_0< \rho(\boldsymbol{a})$ will be chosen later on.
That is, we request the function $t\mapsto \gamma(t) G(t,\sigma(t))$ to have derivative identical to $0$. Thanks to Lemma~\ref{L4}, this
 translates into
\begin{align*}
 & \gamma'(t) G(t,\sigma(t)) + \gamma(t) \partial_t G(t,\sigma(t)) +  \gamma(t) \sigma'(t) \partial_s G(t,\sigma(t))\\
  &= G(t,\sigma(t))\left(\gamma'(t) +  \varphi(t) \gamma(t)\right)
   + \gamma(t)\partial_sG(t,\sigma(t)) \left(\sigma'(t) + q + \varphi(t)\sigma(t)  \right)
 \\ &= 0.
\end{align*}
We deduce that \eqref{eqn:fs} is satisfied provided that $\gamma$ and $\sigma$ solve the  differential equations
\[
  \gamma' + \varphi \gamma = 0 \quad \text{and} \quad \sigma' + \varphi \sigma +q = 0.
\]
In the notation of the statement, we may thus take
$$\gamma(t) =1/A' (t)\quad \text{and} \quad \sigma(t)= q \int_t^{t_0} \exp\left(\int_t^s \varphi(u)\dd u\right)\dd s. $$
We now choose $t_0>0$ sufficiently small so that  $\sigma(t)< 1/\|{\boldsymbol{M}}(\boldsymbol{a},t)\|_{\infty}$ for all $t\leq t_0$.
Observe that $\sigma(0)=q A(t_0)$ and $\sigma(t_0)=0$,
and we then  get from \eqref{eqn:fs}
\[
  G(t_0,0) = A'(t_0) G(0,qA(t_0)).
\]
This allows us to rewrite the identity
\[
   \varphi(t_0) = (1-q)G(t_0,0)- 1 =(1-q)A'(t_0)G(0,qA(t_0)) - 1.
\]
Since $\varphi=A''/A'$, we conclude from  \eqref{E:G(0,s)}  that $A$ satisfies \eqref{eqn:diffA}.
\end{proof}

We can now  solve the Cauchy problem \eqref{eqn:diffA}. Introduce first
\begin{equation} \label{E:pilambda}\Pi_{\boldsymbol{a}}(x)\coloneqq  \prod (1-x{a}_j)^{\nu(j)(1-q) /q},  \qquad 0\leq x \leq 1/ \|\boldsymbol{a}\|_{\infty},
\end{equation}
and its integral
\begin{equation} \label{E:Ilambda}
I_{\boldsymbol{a}}(x)\coloneqq  \int_0^x \Pi_{\boldsymbol{a}}(y)\dd y,  \qquad 0\leq x \leq 1/ \|\boldsymbol{a}\|_{\infty}.
\end{equation}
If we set
$$i_{\boldsymbol{a}}\coloneqq I_{\boldsymbol{a}}(1/ \|\boldsymbol{a}\|_{\infty}),$$
then $I_{\boldsymbol{a}}$ defines a bijection from $[0, 1/ \|\boldsymbol{a}\|_{\infty}]$ to $[0,i_{\boldsymbol{a}}]$, and we write
$$I_{\boldsymbol{a}}^{-1}: [0,i_{\boldsymbol{a}}] \to [0, 1/ \|\boldsymbol{a}\|_{\infty}]$$
for the reciprocal bijection.

\begin{proposition} \label{P1} The explosion time can be identified as
$$\rho(\boldsymbol{a}) = \left\{
\begin{matrix}
 -\log(1-i_{\boldsymbol{a}}/q) & \text{if }i_{\boldsymbol{a}}<q,\\
\infty & \text{if }i_{\boldsymbol{a}}\geq q,
\end{matrix} \right.$$
and for $x\in[0,\rho(\boldsymbol{a}))$, we have
\begin{equation}\label{E:A}
qA(x) =   I^{-1}_{\boldsymbol{a}}( q(1-\e^{-x})).
\end{equation}
\end{proposition}

\begin{proof}  Write  $A^{-1}$ for the reciprocal  of the bijection $A:  [0,\rho(\boldsymbol{a}))\to [0,A(\rho(\boldsymbol{a})))$,
and set  $H\coloneqq A'\circ A^{-1}$, so $H$ is the inverse of the derivative of $A^{-1}$. Using  \eqref{eqn:diffA},  we compute the derivate of $H$ on some small neighborhood of the origin as
\begin{equation} \label{E:H'}
H' = \frac{A'' \circ A^{-1}}{H} = HF - 1.
\end{equation}
Plainly, we have also  the initial value $H(0)=1$.

It is now straightforward to solve this first order linear ODE for $H$. The function $F$ defined in  \eqref{E:F} bears a close connexion to the logarithmic derivative of the function $ \Pi_{\boldsymbol{a}}$ in \eqref{E:pilambda}, namely
$$F(x) =-\frac{q\Pi'_{\boldsymbol{a}}(qx)}{\Pi_{\boldsymbol{a}}(qx)}.$$
Using the notation \eqref{E:Ilambda}, this yields
\begin{equation} \label{E:H}
H(x) =  \frac{ q- I_{\boldsymbol{a}}(qx)}{q\Pi_{\boldsymbol{a}}(qx)};
\end{equation}
more precisely, this identity is valid for all $x>0$ sufficiently small.

Now observe from \eqref{E:H} that $1/H(x) = (A^{-1})'(x)$ is the derivative of
$$B(x)\coloneqq -\log \left( 1- I_{\boldsymbol{a}}(qx)/q\right), \qquad 0\leq x < 1/(q\|{\boldsymbol{a}} \|_{\infty}).$$
The functions  $A^{-1}$ and $B$ both vanish at $x=0$ and have the same derivative on some right neighborhood of $0$, on which they therefore coincide.
Computing the reciprocal function $B^{-1}$, we arrive at \eqref{E:A}, provided that  $x>0$ is sufficiently small.
Tracing the definition, we infer from Lemma~\ref{L2} that $A$ is analytic on $[0,\rho(\boldsymbol{a}))$,
and $I^{-1}_{\boldsymbol{a}}$ is also an analytic function on $[0,i_{\boldsymbol{a}}]$.
The radii of convergence of $A$ and $B^{-1}$ are the same, which yields the  formula for the explosion time. We complete the proof by uniqueness of analytic continuation.
\end{proof}

It is now an easy matter to express the factorial moment generating function ${\boldsymbol{M}}(\boldsymbol{a}, \cdot)$ in terms of the function $A$ that has just been determined.

\begin{corollary}\label{C:Mell} For every  $\boldsymbol{a}$ and $\ell\in \mathrm{Supp}(\boldsymbol{\nu})$ with $a_\ell>0$, we have
$$M_\ell(\boldsymbol{a},t) =  \frac{a_\ell A'(t) }{ 1 - q a_\ell A(t)}, \qquad   t\in [0,\rho(\boldsymbol{a})).$$
\end{corollary}
\begin{proof}   We re-visit the equation in Lemma \ref{L3} and now  view it as a Bernoulli differential equation involving the known function $\varphi$. The latter can be solved explicitly using the transformation $y(t)=1/M_\ell(\boldsymbol{a},t)$; this  yields the linear ODE $$-y'=q+\varphi y.$$
The solution for the initial condition $y(0)=1/a_\ell$ is given by
$$y(t) =  \frac{a_\ell \exp\left(\int_0^t \varphi(s)\dd s\right) }{ 1 - q a_\ell \int_0^t \left(  \exp\left(\int_0^r \varphi(s)\dd s\right) \right) \dd r}, \qquad   t\in [0,\rho(\boldsymbol{a})).$$
Recalling the notation in Lemma \ref{L:A}, we arrive at the stated formula.
\end{proof}

\section{Singularity analysis of the critical case}
This section focuses on the critical case where
\begin{equation} \label{E:crit}
i_{\boldsymbol{a}}=q,
\end{equation}
so by Proposition~\ref{P1}, the explosion time $\rho(\boldsymbol{a})$ is infinite.
 For the sake of simplicity, we also request that
 there exists a unique $j_1\in \mathrm{Supp}(\boldsymbol{\nu})$ such that ${a}_{j_1}= \|\boldsymbol{a}\|_{\infty}$, i.e.
\begin{equation}\label{E:j_1}
{a}_{j_1}>{a}_j\quad \text{ for all }j\neq j_1.\end{equation}
 Although the situation where the maximum is attained for two or more indices can be treated similarly, only
 \eqref{E:j_1} will be relevant for the proof of Theorem~\ref{T1}.

Using the notation
 $\log: \C\backslash \R_- \to \C$ for the principal determination of the complex logarithm and
$$\mathbb{D}\coloneqq \{z\in \C: |z|< 1\}$$
for  the open unit disk in $\C$,
we know from \eqref{E:varphi}
and the proof of  Lemma~\ref{L2} that
$\varphi(-\log(1-z))$ defines an analytic function on $\mathbb{D}$. In short, our main purpose is
show that $\Phi$ can be extended analytically to a domain
$$\Delta(R,\theta)\coloneqq \{z\in \C: |z|< R, \, z\neq 1 \text{ and }|\arg(z-1)|>\theta\}$$
for some $R>1$ and acute angle $\theta \in(0,\pi/2)$, and
 to analyze the asymptotic behavior as $z$ approaches $1$  in $\Delta$.

 \begin{proposition}\label{P2} The following holds under the assumptions \eqref{E:crit} and \eqref{E:j_1}.
 For any acute angle $\theta \in(0,\pi/2)$, there exists some $R=R(\theta)>1$ such that the function $z\mapsto\varphi(-\log(1-z))$ can be extended analytically to $\Delta(R,\theta)$. Moreover,   as $z$ approaches $1$ in $\Delta(R,\theta)$, one has
 $$ \varphi(-\log(1-z)) + 1/\beta  = O\left(|1-z|^{1/\beta}\right),$$
 where
 \begin{equation}\label{E:beta}
 \beta\coloneqq 1+(1-q)\nu(j_1)/q >1.
 \end{equation}

 \end{proposition}

 The most of rest of this section is devoted to the proof of Proposition~\ref{P2}; we will also present some direct consequences at the end.  The assumptions \eqref{E:crit} and \eqref{E:j_1} being implicitly enforced. We further write
 $$a_{j_2}\coloneqq \max\{a_j: j\in \mathrm{Supp}(\boldsymbol{\nu}), j\neq j_1\} < a_{j_1}.$$
Introduce the complex planes slitted along the real half-line $[1/ a_{j_i}, \infty)$,
$$\C_i\coloneqq \{z\in \C: \Im z \neq 0 \text{ or }
\Re z < 1/ a_{j_i}\}, \qquad i=1,2.$$
We now view
$\Pi_{\boldsymbol{a}}$ defined in \eqref{E:pilambda} as a holomorphic function on $\C_1$ that has no zeros on this domain.
Recalling \eqref{E:Ilambda}, we also write  $I_{\boldsymbol{a}}$ for its primitive on $\C_1$ with $I_{\boldsymbol{a}}(0)=0$. The following elementary facts are the keys of our analysis.

\begin{lemma} \label{L5} (i)
The function
$$\frac{q - I_{\boldsymbol{a}}(z)}{q\Pi_{\boldsymbol{a}}(z)}, \qquad z \in \C_1$$
can be extended analytically to $\C_2$. Moreover one has
$$\frac{q - I_{\boldsymbol{a}}(z)}{q\Pi_{\boldsymbol{a}}(z)} \sim \frac{ 1-z a_{j_1}}{a_{j_1}(q+(1-q)\nu(j_1))}
\qquad\text{as } z\to 1/a_{j_1}.$$

(ii) For some $\epsilon>0$ sufficiently small and any acute  angle $\theta\in (0,\pi/2)$,  there is an  open set $\Theta\subset \C_1$ that contains a small open real segment with right extremity $1/a_{j_1}$, such that  the restriction
$$\frac{1}{q}I_{\boldsymbol{a}}: \Theta \to \{z\in \C: 0< |z-1| <\epsilon \text{ and }|\mathrm{Arg}(z-1)| >\theta \}$$
is bijective.
\end{lemma}

\begin{figure}[h]
\centering
\begin{tikzpicture}
\draw[->] (0,0) -- (4,0);
\draw[->] (.5,-1) -- (.5,1.5);
\draw (.5,0) node[below left] {0};
\draw (.6,1) -- (.4,1) node[left] {$i$};
\draw (1.5,.1)--(1.5,-.1) ;
\draw (1.5,0)node[below left] {1};

\draw[thick] (2.3,-.2) arc (-90:90:0.2) -- (1.8,.2) arc (90:270:.2) -- cycle;
\draw (2.05,.5) node {$\Theta$};
\draw[very thick, color=black!50] (2.5,0) -- (2.1,0);
\draw (2.5,0) node {$\times$} node[above right] {$\frac{1}{a_{j_1}}$};
\draw (3.5,0) node{$\times$} node[above] {$\frac{1}{a_{j_2}}$};

\draw[->] (4.5,0) -- (5.5,0) node[above] {$\frac{1}{q}I_{\boldsymbol{a}}$} -- (6.5,0);

\draw[->] (7,0) -- (10,0);
\draw[->] (8,-1) -- (8,1.5);
\draw (8,0) node[below left] {0};
\draw (8.1,1) -- (7.9,1) node[left] {$i$};
\draw (9,.1)--(9,-.1) ;
\draw (9,0)node[below left] {1};
\draw[thick] (9,0) -- (9.4243,0.4243) arc (45:315:0.6) -- cycle;

\draw (9.4,0) node[above right] {$\theta$} arc (0:45:0.4);

\end{tikzpicture}
\caption{{Representation of the domain $\Theta$ and its image via the function $\frac{1}{q}I_{\boldsymbol{a}}$.}}
\end{figure}
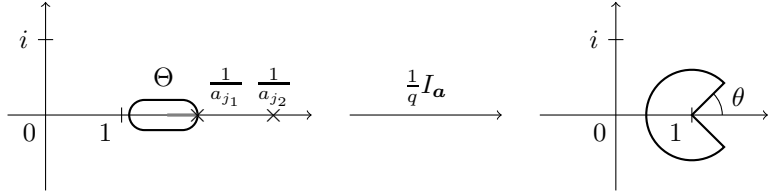

\begin{proof}  We first  make some observations which will be used for establishing each statement.
  Denote the open disk in $\C$ centered at  $1/a_{j_1}$ with radius $ r_1\coloneqq 1/{a_{j_2}}-1/{a_{j_1}}$ by $D(1/a_{j_1},r_1)$, so that $\C_2=\C_1\cup D(1/a_{j_1},r_1)$.

Using the notation  \eqref{E:beta}, we
express  the function $\Pi_{\boldsymbol{a}}$ as the product
$$\Pi_{\boldsymbol{a}}(z)=(1-z a_{j_1})^{\beta-1} \times \Pi^*_{\boldsymbol{a}}(z),$$
where
$$ \Pi^*_{\boldsymbol{a}}(z)\coloneqq   \prod_{j\neq j_1}(1-z{a}_j)^{\nu(j)(1-q) /q}
$$
is a holomorphic function on $\C_2$ which has no zeros there.
Thus $ \Pi^*_{\boldsymbol{a}}$ is  given in the neighborhood of $1/a_{j_1}$ by
$$ \Pi^*_{\boldsymbol{a}}(z)=\sum_{n=0}^{\infty} b_n (1-z a_{j_1})^n, \qquad z\in  D(1/a_{j_1},r_1),$$
for some sequence $(b_n)_{n\geq 0}$ of real numbers with
$$b_0 =   \Pi^*_{\boldsymbol{a}}(1/a_{j_1}) >0\quad\text{and} \quad \limsup_{n\to \infty} |b_n|^{1/n} \leq 1/r_1.$$

On the other hand, we have from the criticality assumption \eqref{E:crit} that
$$q-I_{\boldsymbol{a}}(z) =\int_{[z,1/ a_{j_1}]} \Pi_{\boldsymbol{a}}(z') \dd z',
\qquad z\in \C_1,$$
where we wrote $[z,1/ a_{j_1}]$ for the segment from  $z$ to  $1/ a_{j_1}$, which stays in  $ \C_1$ except for its right-extremity.

(i)  It follows from the observations above that for any $z\in D(1/a_{j_1},r_1)$,
\begin{align}  \label{E:estI}
q-I_{\boldsymbol{a}}(z) &= \sum_{n=0}^{\infty} b_n \int_{[z,1/ a_{j_1}]} (1-z' a_{j_1})^{n+(1-q)\nu(j_1)/q} \dd z'  \nonumber \\
&=  h(z)(1-za_{j_1})^\beta,
 \end{align}
 where
 $$h(z)\coloneqq \frac{1}{a_{j_1}}
\sum_{n=0}^{\infty}
   \frac{b_n }{1+{n+(1-q)\nu(j_1)/q}} (1-za_{j_1})^{n}$$
is a holomorphic function on $D(1/a_{j_1},r_1)$.

We can now write
$$\frac{q - I_{\boldsymbol{a}}(z)}{q\Pi_{\boldsymbol{a}}(z)}= (1-za_{j_1})\frac{ h(z)}{q \Pi^*_{\boldsymbol{a}}(z)}, \qquad z\in \C_1 \cap D(1/a_{j_1},r_1),$$
which immediately entails the two claims in (i).

 (ii) Recall \eqref{E:estI}.
Since $h(1/a_{j_1})$ is a strictly positive real number,  we can pick $r_2>0$ sufficiently small such that
$g: z\mapsto (1-za_{j_1})h(z)^{1/\beta}$ is a well-defined holomorphic function on the disk $D(a_{j_1},r_2)$. Furthermore, since $g(1/a_{j_1})=0$ and $g'(1/a_{j_1})\neq 0$, $g$ is  injective provided that $r_2$ has been chosen sufficiently small, and
 its image contains a small disk around the origin and radius, say  $r_3>0$.
Take an arbitrary angle $\alpha\in(0,\pi)$, and let
$$ \Theta\coloneqq \{g^{-1}(z'): 0<|z'| < r_3 \hbox{ and } |\mathrm{Arg}(z')|<\alpha/\beta\}.$$
Since the real function  $q - I_{\boldsymbol{a}}: [0,1/a_{j_1}] \to [0,q]$ is decreasing bijection,
$\Theta$ contains a real open segment with right-extremity $1/a_{j_1}$. The claim (ii)
is seen from the identity $q- I_{\boldsymbol{a}}(z)= g(z)^{\beta}$ for all $z\in
\Theta$, taking $\theta=\pi-\alpha$.
\end{proof}

We will also need the following.

\begin{lemma} \label{L:bij}
There is a domain $\Theta'$ with
$$(-\infty, 1/a_{j_1})\subset \Theta' \subset\{z: \Re(z)<1/a_{j_1}\},$$
such that the function
$$I_{\boldsymbol{a}}: \Theta' \to \{w\in \C: \Re(w) <q\}$$
is bijective.
\end{lemma}
\begin{proof}
We consider the differential equation
\begin{equation}\label{E:z'(t)}
z'(t) = i/\Pi_{\boldsymbol{a}}(z(t)),
\end{equation}
our main goal is to check that for any initial condition $z(0)=x$ in $ (-\infty, 1/a_{j_1})$, \eqref{E:z'(t)}
has a well-defined solution  for all times $t\in\R$.
In this direction, note first that the function $i/\Pi_{\boldsymbol{a}}$ is globally Lipschitz-continuous on any domain in $\C_1$ whose distance to the half-line $[1/a_{j_1}, \infty)$ is strictly positive. Thanks to the Cauchy-Lipschitz theorem,  the existence of a global solution  to \eqref{E:z'(t)} defined for all $\in \R$ will be granted if  we can check that  for any local solution,
$\Re(z'(t))$ has the opposite sign of $t$. Indeed,
the function $t\mapsto \Re(z(t))$ is then non-decreasing  on $(-\infty,0]$ and non-increasing on $[0,\infty)$,  and hence the distance of $z(t)$ to the half-line $[1/a_{j_1}, \infty)$ is never less than $1/a_{j_1}-x>0$.
For the sake of simplicity we focus on the case when $t>0$ in the sequel, as the case when $t<0$ follows from a similar argument.

In this direction, we consider the continuous determination of the argument of $\Pi_{\boldsymbol{a}}$ on $\C_1$, which we denote by $\arg_{\Pi}$. Explicitly, writing $\arg(z)=\Im(\log(z))\in (-\pi,\pi)$ for the argument of $z\in \C\backslash \R_-$, we have
$$\arg_{\Pi}(z)= \frac{1-q}{q} \sum_{j}  \nu(a_j)\arg(1/a_j - z).$$
Note that $\arg_{\Pi}$ vanishes on $(-\infty,1/a_{j_1})$
and  that $\arg_{\Pi}(z)<0$ for all $z\in \C_1$ with $\Im(z)>0$.  By \eqref{E:z'(t)}, the function $t\mapsto \pi/2- \arg_{\Pi}(z(t))$ is the continuous determination of the argument of $z'(t)$,  so in order to check that
$\Re(z'(t))< 0$, it  suffices to prove that $\arg_{\Pi}(z(t))$ remains bounded from above by $0$ and from below by $-\pi/2$ for any $t>0$.

That  $\arg_{\Pi}(z(t))$ is never $0$ for any $t>0$ is clear, since the  zero set of  $\arg_{\Pi}$ is the  real half-line $(-\infty,1/a_{j_1})$. Assume that $\arg_{\Pi}(z(t))=-\pi/2$ for some $t>0$, and let $t_0>0$ be the smallest of such times.
Then  $\arg (z'(t_0))=-\pi$, that is the derivative $z'(t_0)$ is a strictly negative real number.
Since $\Re(z(t_0))<1/a_j$ for all $j\in \mathrm{Supp}(\boldsymbol{\nu})$, it follows that
$\arg(1/a_j - z(t))>\arg(1/a_j - z(t_0))$ for all $t<t_0$ sufficiently close to $t_0$, and \textit{a fortiori}
$\arg_{\Pi}(z(t))>\arg_{\Pi}(z(t_0))=-\pi/2$ for such $t$'s. This contradicts our assumption, since then there woud exist $t<t_0$ with $\arg_{\Pi}(z(t))=-\pi/2$ . Thus $\arg (z'(t))\in [\pi/2,\pi]$ for all $t\geq 0$, and we conclude that \eqref{E:z'(t)} with arbitrary initial condition  $z(0)=x\in(-\infty,1/a_{j_1})$
 has indeed a solution for all times $t\in \R$.

 Our motivation for introducing \eqref{E:z'(t)} is that its integrated version can be expressed as
\begin{equation} \label{E:z(t)}
I_{\boldsymbol{a}}(z(t))=I_{\boldsymbol{a}}(x)+i t .
\end{equation}
Since
plainly, the function $ I_{\boldsymbol{a}}$ is bijective from
$(-\infty, 1/a_{j_1})$ to $(-\infty, q)$,
we now see that the function from $(-\infty,1/a_{j_1})\times \R$ to $\C_1$ which maps a pair $(x,t)$ to the value $z(t)$ of solution to \eqref{E:z'(t)} at time $t$ for the  initial condition $x$, is injective.
Moreover, again by an argument of Lipschitz-continuity, this function is also continuous,
and we denote its range by
$$\Theta'\coloneqq \{z(t): t\in \R \text{ and }z(0)=x\in (-\infty, 1/a_{j_1})\}.$$
The differential flow \eqref{E:z'(t)} thus induces a bijection between
$(-\infty,1/a_{j_1})\times \R$ and $\Theta'$.
For any  $w\in \C$ with $\Re(w)<q$,
we can choose $x\in (-\infty, 1/a_{j_1})$ such that $I_{\boldsymbol{a}}(x)=\Re(w)$ and $t=\Im(w)$,
and then $I_{\boldsymbol{a}}(z(t))=w$. We conclude from \eqref{E:z(t)} that
the map $ I_{\boldsymbol{a}}: \Theta'\to \{w\in \C: \Re(w)<q\}$ is bijective.
\end{proof}

It is now an easy matter to establish Proposition~\ref{P2} by combining Lemmas~\ref{L5} and \ref{L:bij} with the expressions obtained in the preceding section.

\begin{proof}[Proof of Proposition~\ref{P2}] To start with, we use \eqref{E:H'} to write
$$\varphi(t) = \frac{A''(t)}{A'(t)}= H'(A(t))= H(A(t))F(A(t))-1, \qquad 0\leq t \leq A(\infty)=1/(qa_{j_1}).$$
Applying  Proposition~\ref{P1} at the second line below yields for $x\in[0,1)$
\begin{align*} \varphi(-\log(1-x))&=
 H(A(-\log(1-x)))F(A(-\log(1-x)))-1\\
&=
H(I^{-1}_{\boldsymbol{a}}(qx)/q))F(I^{-1}_{\boldsymbol{a}}(qx)/q))-1.
\end{align*}

On the one hand, recall that  the left-hand side  above defines a holomorphic function on the open unit disk $ \mathbb{D}$. On the other hand,  we know from Lemma \ref{L:bij} that
$I^{-1}_{\boldsymbol{a}}$ can be extended analytically to the half-plane $\{w\in \C: \Re(w)<q\}$ and then takes values in $\{z\in \C: \Re(z) < 1/a_{j_1}\}\subset \C_1$.
Moreover,  we see from \eqref{E:F} and  \eqref{E:H} that $HF$ is  holomorphic on $\{z\in \C: q a_{j_1}\Re(z)<1\}$, and we conclude that the function $z\mapsto\varphi(-\log(1-z))$ can be extended analytically to $\{z\in \C: \Re(z)<1\}$.

Similarly, we know from
Lemma~\ref{L5}(ii)
that, given an acute angle $\theta\in(0,\pi/2)$, any $z$  with $|\mathrm{Arg}(z-1)|>\theta$ and $|z-1|$ sufficiently small, can be given in the form $z=1-I_{\boldsymbol{a}}(z')/q$
for a unique $z'\in \Theta$. On the other hand, we infer from \eqref{E:F}, \eqref{E:H} and Lemma~\ref{L5}(ii) that the function $z'\mapsto H(z'/q)F(z'/q)$ is  holomorphic on $\Theta$, and since $\Theta$ contains an open real segment at the right of $1/a_{j_1}$, we deduce by uniqueness of holomorphic extensions that there is the identity
$$\varphi(-\log(1-I_{\boldsymbol{a}}(z')/q))
= H(z'/q)F(z'/q)-1, \qquad z'\in \Theta.
$$
This establishes the first claim, since we can choose $R>1$ sufficiently close to $1$ such that the domain $\Delta(R,\theta)$ is contained into the union
$$ \{z\in \C: \Re(z)< 1\} \cup  \{z\in \C: 0< |z-1| <\epsilon \text{ and }|\mathrm{Arg}(z-1)| >\theta \} .$$

Finally, we express the right-hand side above using  \eqref{E:F} and \eqref{E:H} as
$$(1-q) \frac{ q- I_{\boldsymbol{a}}(z')}{q\Pi_{\boldsymbol{a}}(z')}\sum_{j} \frac{ {a}_j \nu(j)}{1- z' {a}_j} -1= -\frac{ q}{q+(1-q)\nu(j_1)} + O(|1-z'a_{j_1}|),$$
where the identity stems from Lemma~\ref{L5}(i). Since, thanks to \eqref{E:estI},
$$|1-I_{\boldsymbol{a}}(z')/q|  \asymp |z'a_{j_1}-1|^{\beta} \qquad \text{as }z'\to 1/a_{j_1} \text{ and } z'\in \Theta,
$$
where the notation $f \asymp g$ means that both $f=O(g)$ and $g= O(f)$,
the proof is complete.
\end{proof}

 \section{Proof of Theorem~\ref{T1}}
We now take
$$a_j\coloneqq cj, \qquad j\in \mathrm{Supp}(\boldsymbol{\nu}),$$
where $c>0$ will be chosen later on. Plainly, \eqref{E:j_1} holds with $j_1=\mathrm{k}^*$ and  $a_{j_1}= c\mathrm{k}^*$.
Furthermore, comparing \eqref{E:pi} and \eqref{E:pilambda}, we see that
$$\Pi_{\boldsymbol{a}}(t)= \Pi(ct), \qquad 0\leq t \leq 1/{a_{j_1}},$$
and then \eqref{E:Ilambda} reads
$$I_{\boldsymbol{a}}(1/{a_{j_1}})= \frac{1}{c} \int_0^{1/\mathrm{k}^*} \Pi(t) \dd t.$$
We now choose, in the notation \eqref{E:I},
$$c = 1/m_{\boldsymbol{\nu},q} = \frac{1}{q}\int_0^{1/\mathrm{k}^*} \Pi(t) \dd t,$$
so that \eqref{E:crit} holds.

It is convenient to introduce here the notation $\boldsymbol{\Phi}=(\Phi_{\ell})$, where for $\ell\in \mathrm{Supp}(\boldsymbol{\nu})$,
\begin{equation}\label{E:Phi}\Phi_\ell(z)\coloneqq M_\ell(\boldsymbol{a},-\log(1-z)).
\end{equation}
We set further
$$\Phi_{ \boldsymbol{\nu}} \coloneqq \left\langle \boldsymbol{\nu}; \boldsymbol{\Phi}\right\rangle,$$
so that, by \eqref{E:varphi},
\begin{equation} \label{E:Phiphi}
\Phi_{ \boldsymbol{\nu}} (z)= \frac{\varphi(-\log (1-z))+1}{1-q}.
\end{equation}

Next, recall from Lemma~\ref{L1} that for any $x\in [0,1)$,
\begin{align*}\Phi_{ \boldsymbol{\nu}}(x) &= \EE\left( \prod (cj)^{Y_j(-\log(1-x))}\right) \\
&= (1-x)\sum_{n=1}^{\infty} x^{n-1} c^n\E(Z(n)).
\end{align*}
We have seen that this power series has radius of convergence equal to $1$ and know from Proposition~\ref{P2}  that it defines a holomorphic function on some domain $\Delta(R, \theta)$.

We next re-express the above in the form
\begin{align*}& \frac{\Phi_{ \boldsymbol{\nu}}(z)
-  \nu(j_1)/(q+(1-q)\nu(j_1)) }{1-z}
\\
&= \sum_{n=1}^{\infty} z^{n-1} \left( m_{\boldsymbol{\nu},q}^{-n}\E(Z(n))- \frac{ \nu(j_1)}{q+(1-q)\nu(j_1)}
\right).
\end{align*}

The assumptions of Proposition~\ref{P2} have been checked, and combining its conclusion with \eqref{E:Phiphi}  and the observation from \eqref{E:beta} that
$$ \frac{ \nu(j_1)}{q+(1-q)\nu(j_1)}= \frac{1-1/\beta}{1-q},
 $$
 enables us to apply  to this  generating function a basic transfer theorem of the singularity analysis; see  \cite[Theorem 1]{FlajOdl}, or  \cite[Theorem VI.3]{FlajSed}.
 We get
$$m_{\boldsymbol{\nu},q}^{-n}\E(Z(n))- \frac{ \nu(j_1)}{q+(1-q)\nu(j_1)} = O(n^{-1/\beta}).$$
Recalling that $j_1=\mathrm{k}^*$, this is the claim of Theorem~\ref{T1}.

\section{Three further results}

\subsection{On the dependence on the parameters}
In random population models, the ratio $\E(Z(n+1))/\E(Z(n))$  is sometimes called the effective reproduction number at generation $n$, and Theorem~\ref{T1} thus identifies its limit when $n\to \infty$  for the reinforced Galton-Watson process as $m_{\boldsymbol{\nu},q}$ in \eqref{E:I}. It is natural to study how this quantity depends on the memory parameter $q$ and on the reproduction law $\boldsymbol{\nu}$.

\begin{proposition}\label{P3} (i) The function $q\mapsto m_{\boldsymbol{\nu},q}$ is monotone increasing on $(0,q)$,  with
$$\lim_{q\to 0+} m_{\boldsymbol{\nu},q}= \sum j \nu(j) \quad\text{and}  \quad \lim_{q\to 1-} m_{\boldsymbol{\nu},q}=\mathrm{k}^*.$$

\noindent (ii) Let $q\in(0,1)$ be a fixed memory parameter and $\mathrm{k}^*\geq 1$ a fixed integer.
The function $\boldsymbol{\nu} \mapsto  m_{\boldsymbol{\nu},q}$ is log-concave on the subset of
reproduction laws $\boldsymbol{\nu}$ with
$ \sup\{k\geq 0: \nu(k)>0\}= \mathrm{k}^*$.
\end{proposition}
\begin{proof} (i)  The monotonicity assertion derives from Lemma~\ref{L-1} and Theorem~\ref{T1}.
Next, we write
\begin{equation} \label{E:1/m}
1/m_{\boldsymbol{\nu},q}= \int_0^{1/(q\mathrm{k}^*)} \Pi(qt)\dd t,
\end{equation}
and note that for every $t\geq 0$,
$$\lim_{q\to 0+} \Pi(qt) = \exp\left(-t\sum j\nu(j)\right).$$
Integrating over $t\geq 0$ yields
$$\lim_{q\to 0+} 1/m_{\boldsymbol{\nu},q} = \left. 1\middle/\sum j \nu(j)\right. .$$
On the other hand,
the function
$$q \mapsto (1-tk)^{\nu(k)(1-q) /q} $$
increases on $(0,1)$. Its limit as $q\to 1-$ equals $1$ for all $t <1/\mathrm{k}^*$, which yields
$m_{\boldsymbol{\nu},1-}=\mathrm{k}^*$.

(ii) Let $\boldsymbol{\nu}_1$ and $\boldsymbol{\nu}_2$ be two reproduction laws
with $ \sup\{k\geq 0: \nu_i(k)>0\}= \mathrm{k}^*$ for $i=1,2$.
Write
$$\Pi_i(t)\coloneqq  \prod (1-tk)^{\nu_i(k)(1-q) /q},  \qquad t\leq 1/\mathrm{k}^* \text{ and }i=1,2.$$
Let $c\in(0,1)$, consider $\boldsymbol{\nu}= c \boldsymbol{\nu}_1+(1-c)\boldsymbol{\nu}_2$.  As $\Pi=\Pi_1^c \times \Pi_2^{1-c}$ in the obvious notation,
we deduce from \eqref{E:1/m} and the H\"older inequality  that
$$m_{\boldsymbol{\nu},q}^{-1} \leq m_{\boldsymbol{\nu}_1,q}^{-c} m_{\boldsymbol{\nu}_2,q}^{c-1}.$$
This shows the log-concavity statement.
\end{proof}

\subsection{Explicit upper and lower bounds}
The purpose of this section is to discuss the following simple bounds for the  asymptotic  effective reproduction number $m_{\boldsymbol{\nu},q}$.

\begin{proposition} \label{P5} We have
$$\mathrm{k}^*(q+(1-q)\nu(\mathrm{k}^*)) \leq m_{\boldsymbol{\nu},q} \leq \mathrm{k}^*q +(1-q) \bar {\boldsymbol{\nu}},$$
where in the right-hand side,
$$\bar {\boldsymbol{\nu}}\coloneqq \sum j \nu(j)$$
denotes the mean of ${\boldsymbol{\nu}}$.
\end{proposition}
We shall now give two different proofs of this claim, the first is purely probabilist whereas the second is purely analytic and  based on the formula \eqref{E:I}.

\begin{proof}[A probabilistic proof]
 Consider first the sub-process of the reinforced Galton-Watson process which results by suppressing every sibling (together with its possible descent) which has  size strictly less than $\mathrm{k}^*$. This sub-process is then a true Galton-Watson process in which individuals have either $\mathrm{k}^*$ children with probability $q+(1-q)\nu(\mathrm{k}^*)$ or $0$ child with complementary probability. Therefore $Z(n)$ dominates the size of the $n$-th generation of a Galton-Watson process with averaged reproduction $\mathrm{k}^*(q+(1-q)\nu(\mathrm{k}^*))$, and the lower-bound in the statement follows.

 For the upper-bound consider another Galton-Walton process with the following reproduction law. With probability $q$, the number of children of a typical individual is $\mathrm{k}^*$, and with complementary probability $1-q$, it is given by a random variable with law $\boldsymbol{\nu}$. In words, at each repetition event, we systematically increase the number of children to $\mathrm{k}^*$. Obviously, this second true Galton-Watson process dominates the reinforced Galton-Watson process, and this entails the upper-bound.
 \end{proof}

\begin{proof} [An analytic proof]
For the lower-bound, we simply write for every $0\leq t \leq 1$ the inequality
$$\prod \left(1-tj/\mathrm{k}^*\right)^{(1-q)\nu(j)/q} \leq  \left(1-t\right)^{(1-q)\nu(\mathrm{k}^*)/q}.
$$
By integration, we conclude that
$$\int_0^1\prod \left(1-tj/\mathrm{k}^*\right)^{(1-q)\nu(j)/q} \dd t \leq
 \frac{q}{q+(1-q) \nu(\mathrm{k}^*)},
$$
and now the first inequality of the statement can be seen from \eqref{E:I}.

For the upper-bound, we first write
$$\left(1-tj/\mathrm{k}^*\right)^{(1-q)\nu(j)/q} = \left( \left(1-tj/\mathrm{k}^*\right)^{\mathrm{k}^*/j} \right)^\frac{(1-q)j \nu(j)}{q\mathrm{k}^*} ,$$
and then use the convexity inequality
$$(1-at)^{1/a}\geq 1-t\qquad \text{for }a\leq 1.$$
We get
$$\prod \left(1-tj/\mathrm{k}^*\right)^{(1-q)\nu(j)/q} \geq  \left(1-t\right)^\frac{(1-q)\bar{\boldsymbol{\nu}}}{q\mathrm{k}^*}.
$$
By integration, we get
$$\int_0^1\prod \left(1-tj/\mathrm{k}^*\right)^{(1-q)\nu(j)/q} \dd t \geq
 \frac{q}{q+(1-q) \bar{\boldsymbol{\nu}}/\mathrm{k}^*},
$$
and the second inequality of the statement can again be seen from \eqref{E:I}.
\end{proof}

\subsection{On the dependence on the first offspring number}
We now conclude this work presenting a stronger version of Theorem~\ref{T1}. More precisely, the latter is given for the probability measure $\P$
under which the number of children $Z(1)$ of the ancestor is a random sample of $\boldsymbol{\nu}$. Here is the version under the conditional probabilities $\P_{\ell}=\P(\cdot \mid Z(1)=\ell)$, which again underlines the key role of the  maximal possible number of children $\mathrm{k}^*$.

\begin{theorem}\label{T2} There exists some constant $\gamma\in(0,\infty)$ such that
for every $\ell\in \mathrm{Supp}(\boldsymbol{\nu})$ with $\ell\neq \mathrm{k}^*$, we have
$$\lim_{n\to \infty} n^{1/\beta}  m_{\boldsymbol{\nu},q}^{-n}\E_{\ell} (Z(k)) = \frac{\gamma}{\Gamma(1/\beta)(1/\ell-1/\mathrm{k}^*)}.
$$
We have also
$$\lim_{n\to \infty}  m_{\boldsymbol{\nu},q}^{-n}\E_{\mathrm{k}^*} (Z(k)) = 1/(q+
 \nu(\mathrm{k}^*)(1-q)).$$
\end{theorem}
\begin{remark}
One can compute the value of $\gamma$  explicitly in terms of the parameters. However, since its expression is not quite simple and chasing the constants in the calculation below would have been a bit boring, we favored simplicity of the argument over a more precise statement.
\end{remark}

\begin{proof} The guiding line of the argument is similar to that for the proof of Theorem~\ref{T1} and will not be repeated. We shall just indicate the main steps of the calculation.
Recall that we focus on the critical case when assumptions \eqref{E:crit} and \eqref{E:j_1} hold.

Our starting point is now the formula for the factorial moment generating function that been obtained in Corollary \ref{C:Mell}. Using the notation \eqref{E:Phi}, this reads
$$\Phi_\ell(z)= \frac{a_\ell A'(-\log(1-z)) }{ 1 - q a_\ell A(-\log(1-z))}.$$
Essentially, we have to check that  this defines an analytic function on $ \Delta(R,\theta)$ and then to determine its asymptotic behavior as $z\to 1$ in $ \Delta(R,\theta)$.

Let us first discuss analyticity; recall that $A'(t)=\exp\left( \int_0^t \varphi(s)\dd s\right)$. Thanks to Proposition~\ref{P2}, both the numerator and the denominator of $\Phi_\ell(z)$ are analytic functions on $ \Delta(R,\theta)$, so all that is needed is to check that the denominator does not vanish there. Assume that there is some $ z_\ell \in \Delta(R,\theta)$ such that
$ q a_\ell A(-\log(1-z_\ell))=1$. Obviously, then $ q a_j A(-\log(1-z_\ell))\neq 1$ for any $j$ with $a_j\neq a_\ell$, and we deduce that $z_\ell$ would also be a singularity of the linear combination  $\Phi_{ \boldsymbol{\nu}}$. This contradicts Proposition~\ref{P2}(ii) and we conclude that $\Phi_\ell$ is indeed analytic on $\Delta(R,\theta)$ for all $\ell$.

Next, we know from Proposition~\ref{P2} that
$$\varphi(w) = -1/\beta + O(\e^{-w/\beta}) \qquad \text{as } \Re(w)\to \infty \text{ with }1-\e^{-w}\in \Delta(R,\theta).$$
We deduce by integration that in the same regime
$$A'(w)=\exp\left( \int_{[0,w]} \varphi(s)\dd s\right) \sim \gamma \e^{-w/\beta}(1+O(\e^{-w/\beta}),$$
where
$$\gamma\coloneqq \exp\left( \int_0^{\infty} (\varphi(t)+1/\beta)\dd t\right) \in (0,\infty).$$
We now see that for any $\ell\neq j_1$, we have
$$\Phi_\ell(z) \sim \frac{ \gamma (1-z)^{1/\beta}}{1/a_\ell-1/a_{j_1}} \qquad \text{as }z\to 1, z\in \Delta(R,\theta).$$
The rest of the proof of the statement for $\ell \neq \mathrm{k}^*$ uses the same path as in Section 6, except that for the transfer, we use \cite[Corollary VI.1]{FlajSed} in place of \cite[Theorem VI.3]{FlajSed}. The second claim then  follows by linearity from the first and Theorem~\ref{T1}.
\end{proof}

\vskip 3mm
\noindent
{\bf Acknowlegdment:} J.B. would like to thank Klaus Widmayer for discussions about transport equations, and notably Lemma \ref{L3}.


\end{document}